\documentclass[10pt]{amsart}
\usepackage{amssymb}
\usepackage{amsbsy}
\usepackage{amscd}
\usepackage[mathscr]{eucal}
\usepackage{verbatim}
\usepackage{version}
\usepackage{color}
\usepackage[matrix,arrow]{xy}

\def\cal{\mathcal}
\def\Bbb{\mathbb}
\def\frak{\mathfrak}

\newenvironment{NB}{
\color{red}{\bf NB}. \footnotesize 
}{}
\excludeversion{NB}
\newenvironment{NB2}{
\color{blue}{\bf NB}. \footnotesize
}{}

\newcommand{\Pic}{\operatorname{Pic}}
\newcommand{\Quot}{\operatorname{Quot}}
\newcommand{\Supp}{\operatorname{Supp}}
\newcommand{\ch}{\operatorname{ch}}

\newcommand{\Coh}{\operatorname{Coh}}

\newcommand{\Ext}{\operatorname{Ext}}
\newcommand{\Hom}{\operatorname{Hom}}

\newcommand{\im}{\operatorname{im}}

\newcommand{\rk}{\operatorname{rk}}

\newcommand{\NS}{\operatorname{NS}}

\newcommand{\td}{\operatorname{td}}

\newcommand{\Amp}{\operatorname{Amp}}

\newcommand{\tr}{\operatorname{tr}}

\newcommand{\NSf}{\operatorname{NS_{\mathrm{f}}}}

\font\b=cmr10 scaled \magstep5
\def\bigzerou{\smash{\lower1.7ex\hbox{\b 0}}}

\numberwithin{equation}{section}

\theoremstyle{plain}
 \newtheorem{thm}{Theorem}[section]
 \newtheorem{lem}[thm]{Lemma}
 \newtheorem{prop}[thm]{Proposition}
 \newtheorem{cor}[thm]{Corollary}
 
\theoremstyle{definition}
 \newtheorem{defn}[thm]{Definition}
 
\theoremstyle{remark}
 \newtheorem{rem}[thm]{Remark}

\begin{document}

\title{A note on stable sheaves on Enriques 
surfaces}
\author{K\={o}ta Yoshioka}
\address{Department of Mathematics, Faculty of Science,
Kobe University,
Kobe, 657, Japan
}
\email{yoshioka@math.kobe-u.ac.jp}

\thanks{
The author is supported by the Grant-in-aid for 
Scientific Research (No.\ 26287007,\ 24224001), JSPS}

\begin{abstract}
We shall give a necessary and sufficient condition for the 
existence of stable sheaves on Enriques surfaces based on results of
Kim, Yoshioka, Hauzer and Nuer.
For unnodal Enriques surfaces, we also study the relation of virtual
Hodge ``polynomial'' of the moduli stacks.
\end{abstract}

\maketitle

\renewcommand{\thefootnote}{\fnsymbol{footnote}}

\section{Introduction}
Studies of
moduli spaces of stable sheaves on Enriques surfaces were
started by a series of
works of Kim \cite{Kim1},\ \cite{Kim2},\ \cite{Kim3},\ 
\cite{Kim:excep},\ \cite{Kim4}.
In particular, he studied exceptional bundles and the singular locus
of the moduli spaces. Recently the type of singularities are investigated
by Yamada \cite{Yamada}.
For the topological properties of the moduli spaces,
the author \cite{Y:twist1} 
computed the Hodge polynomials of the moduli spaces
if the rank is odd. In particular, the condition for
the non-emptiness of the moduli spaces are known. 
For the even rank case,
by extending our arguments, Hauzer \cite{Hauzer}
related the virtual Hodge ``polynomial'' of the moduli
spaces to those for rank 2 or 4.
Then Nuer \cite{N} gave
the condition for the non-emptiness
by studying the non-emptiness for rank 2 and 4 cases. 
The main purpose of this note is
to give another proof of his result on the non-emptiness.

\begin{thm}\label{thm:exist:intro}
Let $X$ be an unnodal Enriques surface over ${\Bbb C}$.
For $r,s \in {\Bbb Z}$ and $L \in \NS(X)$ such
that $r-s$ is even, 
let ${\cal M}_H(r,L,-\frac{s}{2})$ be the stack of 
semi-stable sheaves $E$ of rank $r>0$, $\det E=L$ and
$\chi(E)=\frac{r-s}{2}$, where the polarization is $H$.
Assume that 
$\gcd(r,c_1(L),\frac{r-s}{2})=1$, i.e., the Mukai vector
is primitive.
Then
${\cal M}_H(r,L,-\tfrac{s}{2}) \ne \emptyset$ for a general $H$
if and only if
\begin{enumerate}
\item
$\gcd(r,c_1(L),s)=1$ and $(c_1(L)^2)+rs \geq -1$ or 
\item 
$\gcd(r,c_1(L),s)=2$ and $(c_1(L)^2)+rs \geq 2$ 
or 
\item
$\gcd(r,c_1(L),s)=2$,
$(c_1(L)^2)+rs =0$ and $L \equiv \frac{r}{2}K_X \mod 2$.
\end{enumerate}
If $r=0$, then by assuming $L$ to be effective,
the same claim holds.
\end{thm}
Since $v$ is primitive and $H$ is general,
semi-stability implies stability.

In order to explain the difference of the proofs,
we first mention the results in \cite{Y:twist1} and \cite{Hauzer}.
In \cite{Y:twist1},
we introduced the virtual Hodge ``polynomial''
$e({\cal M}_H(r,L,-\frac{s}{2}))$ 
of the moduli stacks, which is an extension of the virtual Hodge
polynomial of an algebraic set and showed that 
it is preserved under a special kind of Fourier-Mukai
transform. 
As an application, we showed that 
$e({\cal M}_H(r,L,-\frac{s}{2}))$ 
is the same as 
$e({\cal M}_H(1,0,\frac{1}{2}-n))$ 
if $r$ is odd, where $2n=(c_1(L)^2)+rs+1$
\cite[Thm. 4.6]{Y:twist1}.
In particular we get the condition 
$(c_1(L)^2)+rs \geq -1$ for the non-emptiness.
Hauzer \cite{Hauzer} generalized our method and showed that
$e({\cal M}_H(r,L,-\frac{s}{2}))$ 
is the same as 
$e({\cal M}_H(r',L',-\frac{s'}{2}))$ 
where $r'=2,4$ and $(c_1(L')^2)+r's'=(c_1(L)^2)+rs$.
For the rank 2 case,
the condition of non-emptiness follows
by Kim's results \cite{Kim4}.
Thus the remaining problem is to treat the rank 4 case.

For this problem, Nuer \cite[Thm. 5.1]{N}
constructed $\mu$-stable vector bundles of rank 4
by Serre construction, and got the 
condition for the non-emptiness.
On the other hand, we shall 
reduce the rank 4 case to  
the rank 2 case by improving Hauzer's argument
(Theorem \ref{thm:e-poly}).
Combining Kim's results \cite{Kim4}, 
Theorem \ref{thm:exist:intro} follows.
For convenience sake,
we also give another argument for the rank 2 case
using a relative Fourier-Mukai
transform associated to an elliptic fibration.
Replacing virtual Hodge ``polynomial'' by numbers of ${\Bbb F}_q$-rational
points, our result also holds for unnodal Enriques surfaces
over an algebraically closed field of characteristic $p \ne 2$. 
As a corollary of Theorem \ref{thm:exist:intro},
by adding a deformation argument, 
we shall treat the nodal case in Section \ref{sect:nodal}.

Finally I would like to remark another approach in Appendix.
For our argument, main tool is a special
kind of Fourier-Mukai transforms.
For the case of K3 surfaces, Toda \cite{Toda} proved 
a certain counting invariant of the moduli stack of Bridgeland
semi-stable objects are invariant under Fourier-Mukai transforms.
Since Gieseker stability corresponds to the large volume limit
of Bridgeland stability,
it is possible to get Theorem \ref{thm:exist:intro} 
by a more sophisticated method, i.e., Bridgeland theory of 
stability conditions \cite{Br:3}.  
For a more general treatment,
we recommend a reference \cite{N2}.

\section{Proof of Theorem \ref{thm:exist:intro}}


\subsection{Notation and some tools}

We prepare several notation and results which will be used.

The Mukai vector $v(x)$ of $x \in K(X)$
is defined as an element of $H^*(X,{\Bbb Q})$:
\begin{equation}
\begin{split}
v(x):=&\ch(x)\sqrt{\td_X}\\
=&\rk(x)+c_1(x)+\left(\frac{\rk(x)}{2}\varrho_X+\ch_2(x)\right)
\in H^*(X,{\Bbb Q}),
\end{split}
\end{equation}
where $\varrho_X$ is the fundamental class of $X$.
We also introduce Mukai's pairing on $H^*(X,{\Bbb Q})$ by
$\langle x,y \rangle:=-\int_X x^{\vee} \wedge y$.
Then we have an isomorphism of lattices:
\begin{equation}
(v(K(X)),\langle \;\;,\;\;\rangle) \cong
\begin{pmatrix}
1&0\\
0&-1
\end{pmatrix}
\oplus 
\begin{pmatrix}
0&1\\
1&0
\end{pmatrix}
\oplus E_8(-1).
\end{equation}
\begin{defn}
We call an element of $v(K(X))$ by the Mukai vector.
A Mukai vector $v$ is primitive, if
$v$ is primitive as an element of $v(K(X))$.
\end{defn}
We denote the torsion free quotient of $\NS(X)$ by
$\NSf(X)$, that is, $\NSf(X)=\NS(X)/{\Bbb Z}K_X$.

\begin{lem}\label{lem:Hauzer}
Let $v=(r,c_1,-\frac{s}{2})$ ($r,s \in {\Bbb Z}$, $2\mid r-s$, 
$c_1 \in \NSf(X)$) be a Mukai vector.
 \begin{enumerate}
\item[(1)]
$v$ is primitive if and only if $\gcd(r,c_1,\frac{r-s}{2})=1$.
\item[(2)]
Assume that $v$ is primitive.
We set $\ell:=\gcd(r,c_1,s)$. Then $\ell=1,2$.
\begin{enumerate}
\item
If $\ell=1$, then $\gcd(r,c_1,2)=1$.
\item
If $\ell=2$, then $2 \mid r$, $2 \mid c_1$, $2 \mid s$ and 
$r+s \equiv 2 \mod 4$.
\end{enumerate}
\end{enumerate} 
\end{lem}

\begin{proof}
(1)
For $E=r {\cal O}_X+F \in K(X)$ with $\rk F=0$,
$v(E)=(r,0,\frac{r}{2})+(0,D,t)$, where $D \in \NSf(X)$ and $t \in {\Bbb Z}$.
Then $v(E)$ is primitive if and only if
$\gcd(r,D,t)=1$.
If $v=v(E)$, then $c_1=D$ and $t+\frac{r}{2}=-\frac{s}{2}$.
Hence $\gcd(r,c_1,\frac{r-s}{2})=\gcd(r,D,t)$, which shows the claim.

(2)
It is \cite[Lem.\ 2.5]{Hauzer}.
For convenience sake, we give a proof.
Since $s=r+2\frac{s-r}{2}$,
$\ell=1,2$. If $\ell=1$, then $\gcd(r,c_1,2)=1$.
If $\ell=2$, then $2 \mid r$, $2 \mid c_1$.
Since $\gcd(r,c_1,\frac{s-r}{2})=1$, 
$r+s \equiv 2 \mod 4$.
\end{proof}

For a variety $Y$ over ${\Bbb C}$, the cohomology with compact support
$H^*_c(Y,{\Bbb Q})$ has a natural mixed Hodge structure.
Let $e^{p,q}(Y):=\sum_k(-1)^k h^{p,q}(H_c^k(Y))$ be the virtual Hodge number
and $e(Y):=\sum_{p,q}e^{p,q}(Y)x^p y^q$
the virtual Hodge polynomial of $Y$.

For $\alpha \in \NS(X)_{\Bbb Q}$,
a torsion free sheaf $E$ is $\alpha$-twisted semi-stable with
respect to $H$,
if
\begin{equation}\label{eq:twisted-stability}
\frac{\chi(F(-\alpha+nH))}{\rk F} \leq
\frac{\chi(E(-\alpha+nH))}{\rk E}\;\; (n \gg 0)
\end{equation}
for all subsheaf $F$ of $E$ \cite{MW}.
${\cal M}_H^{\alpha}(v)$ denotes the
moduli stack of $\alpha$-twisted semi-stable sheaves $E$ with $v(E)=v$,
where $H$ is the polarization.
$(H,\alpha)$ is general with respect to $v$, if
equality in \eqref{eq:twisted-stability} implies
$$
\frac{v(F)}{\rk F}=\frac{v(E)}{\rk E}.
$$
In particular, if $v$ is primitive, then ${\cal M}_H^{\alpha}(v)$
consists of $\alpha$-twisted stable objects for a general
pair $(H,\alpha)$.
If $\alpha=0$, then we write ${\cal M}_H(v)$.
Then ${\cal M}_H^{\alpha}(v)$ is described as a quotient stack 
$[Q^{ss}/GL(N)]$, where $Q^{ss}$ is a suitable open 
subscheme of $\Quot_{{\cal O}_X^{\oplus N}/X}$.
We define the virtual Hodge ``polynomial'' of ${\cal M}_H^{\alpha}(v)$
by 
\begin{equation}
e({\cal M}_H^{\alpha}(v))=
e(Q^{ss})/e(GL(N)) \in {\Bbb Q}(x,y).
\end{equation}
It is easy to see that $e(Q^{ss})/e(GL(N))$ does not depend on
the choice of $Q^{ss}$.
The following was essentially proved in \cite[Sect.\ 3.2]{chamber}
(see also \cite[Sect.\ 2.2]{Y:twist2}).
\begin{prop}\label{prop:indep}
Let $X$ be a surface such that $K_X$ is
numerically trivial.
Let $(H,\alpha)$ be a pair of ample divisor $H$
and a ${\Bbb Q}$-divisor $\alpha$.
Then
$e({\cal M}_H^{\alpha}(v))$ does not depend on the choice of
$H$ and $\alpha$, if $(H,\alpha)$ 
is general with respect to $v$.
\end{prop}

\begin{NB}
The same claim holds if we specified the determinant line bundle
$L$.
\end{NB}

By using a special kind of Fourier-Mukai transform
called $(-1)$-reflection and using Proposition \ref{prop:indep}, 
we get the following result.

\begin{prop}[{\cite[Prop.\ 4.5]{Y:twist1}}]\label{prop:hodge-enriques}
Let $X$ be an unnodal Enriques surface.
Assume that $r,s>0$.
Then
\begin{enumerate}
\item[(1)]
$$
e\left({\cal M}_H^{\alpha}\left(r,c_1,-\tfrac{s}{2}\right)\right)=
e\left({\cal M}_H^{\alpha}\left(s,-c_1,-\tfrac{r}{2}\right)\right)
$$
for a general $(H,\alpha)$,
if $(c_1^2)<0$, i.e, $\langle v^2 \rangle<rs$,
where $v=(r,c_1,-\tfrac{s}{2})$.
In particular, if $r> \langle v^2 \rangle$, then
we get our claim.
\item[(2)]
If we specify the first Chern class as an element of
$\Pic(X) \cong \NS(X)$,  then we also have
$$
e\left({\cal M}_H^{\alpha}
\left(r,L+\tfrac{r}{2}K_X,-\tfrac{s}{2}\right)\right)=
e\left({\cal M}_H^{\alpha}
\left(s,-(L+\tfrac{s}{2}K_X),-\tfrac{r}{2}\right)\right)
$$
for a general $(H,\alpha)$,
if $(c_1(L)^2)<0$, i.e, $\langle v^2 \rangle<rs$,
where $v=(r,c_1(L),-\tfrac{s}{2})$.
\end{enumerate}
\end{prop}

\begin{rem}
\begin{enumerate}
\item[(1)]
For the proof of Proposition \ref{prop:hodge-enriques} (2),
we use the description of the $(-1)$-reflection 
as a Fourier-Mukai transform (see Appendix).
Then the first Chern class $L+\frac{r}{2}K_X$
is replaced by $-[(L+\tfrac{r}{2}K_X)+\langle v,v(K_X) \rangle K_X]=
-(L+\tfrac{s}{2}K_X)$.
\item[(2)]
The same claim also holds for nodal case (see Appendix).
\end{enumerate}
\end{rem}
 
\begin{NB}
\begin{proof}
If $(c_1^2)<0$, then
the Hodge index theorem implies that
there is a divisor $H$ such that
$(H,c_1)=0$ and $({H}^2)>0$. By the Riemann-Roch theorem,
we may assume that $H$ is effective. 
Since $X$ is unnodal, $H$ is ample.   
If $E_0={\cal O}_X$, then
$v({\cal E}_{|\{x \} \times X})=2$.
Hence $v$ satisfies assumptions of Proposition \ref{prop:00/08/17}.
Then we get an isomorphism 
\begin{equation}
 {\cal M}_{H}^{{\cal O}_X+\varepsilon}(r+c_1-(s/2) \varrho_X) \to
 {\cal M}_{H}^{{\cal O}_X+\varepsilon}(s-c_1-(r/2) \varrho_X),
\end{equation}
where $(H,{\cal O}_X+\varepsilon)$ is general with respect to $v$.
By Proposition \ref{prop:indep}, we get our claim.
\end{proof}
\end{NB}

\subsection{Reduction to the rank 2 case}\label{subsect:reduction}

From Subsection \ref{subsect:reduction} to
Subsection \ref{subsect:unnodal:rank0},
we assume that $X$ is an unnodal Enriques surface and
$r$ is even (and hence $s$ is also even). 
We also assume that $\alpha=0$, that is,
we consider the moduli stack of ordinary Gieseker semi-stable sheaves
${\cal M}_H(v)$.
We shall prove the following result in this subsection.

\begin{thm}\label{thm:e-poly}
Let $v=(r,c_1,-\frac{s}{2})$ be a primitive Mukai
vector such that $r>0$ is even.
\begin{enumerate}
\item[(1)]
If $\gcd(r,c_1,s)=1$, then
$e({\cal M}_H(r,c_1,-\tfrac{s}{2}))=e({\cal M}_H(2,\xi,-\frac{s'}{2}))$
for a general $H$,
where $\xi$ is a primitive element of $\NSf(X)$ and
$(\xi^2)+2s'=(c_1^2)+rs$.
\item[(2)]
If $\gcd(r,c_1,s)=2$, then
$e({\cal M}_H(r,c_1,-\tfrac{s}{2}))=e({\cal M}_H(2,0,-\frac{s'}{2}))$
for a general $H$,
where $2s'=(c_1^2)+rs$.
\end{enumerate}
\end{thm}

For the proof of this result,
we shall slightly improve Hauzer's argument.
Let ${\Bbb Z}\sigma+{\Bbb Z}f$ be a hyperbolic lattice
in $\NS(X)$: 
$$
(\sigma^2)=(f^2)=1,\;(\sigma,f)=1.
$$
The main difference of \cite{Y:twist1} and \cite{Hauzer} is the case 
${\cal M}_H(r,c_1,-\frac{s}{2})$ such that $r$ is even and
$c_1=\frac{r}{2}bf+\frac{r}{2}b' \sigma+\xi$, $b,b'=0,1$,
$\xi \in E_8(-1)$.
In order to treat this case, we shall modify 
the argument in \cite{Hauzer}.  
For a primitive Mukai vector
$(r,\frac{r}{2}bf+\xi,-\frac{s}{2})$ ($b =0,-1,1, \xi \in E_8(-1)$),
\cite[Cor. 2.6] {Hauzer} implies that 
$\gcd(r,\xi,s)=1,2$.
Indeed $1=\gcd(r,\frac{r}{2}bf+\xi,\frac{r-s}{2})=
\gcd(\frac{r}{2},\frac{s}{2},\xi)$ implies 
$\gcd(r,\xi,s)=1,2$.
\begin{NB}
$\gcd(r,\frac{r}{2}bf+\xi,\frac{r-s}{2})=1$ implies
$\gcd(\frac{r}{2},\frac{s}{2},\xi)=1$.
\end{NB}
\begin{NB}
$\gcd(r,\xi,s)=1$ if $2 \nmid \xi$ and 
$\gcd(r,\xi,a)=2$ if $2 \mod \xi$.
\end{NB}

\begin{lem}\label{lem:reduction1}
For a primitive Mukai vector
$v=(r,\frac{r}{2}bf+\xi,-\frac{s}{2})$ ($b =0,-1,1, \xi \in E_8(-1)$),
we set $l:=\gcd(r,\xi,s)$.
\begin{enumerate}
\item
[(1)]
$e({\cal M}_H(r,\frac{r}{2}bf+\xi,-\frac{s}{2}))=
e({\cal M}_H(r',\frac{r}{2}bf+\xi',-\frac{s'}{2}))$ for a general $H$,
where $r' \equiv r \mod 2l$, $s' \equiv s \mod 2l$,
$l=\gcd(r',\xi',s')$,
$\xi'/l \in E_8(-1)$ is primitive and
$r's' \geq r' >\langle v^2 \rangle$. 
\item
[(2)]
$e({\cal M}_H(r,\frac{r}{2}bf+\xi,-\frac{s}{2}))=
e({\cal M}_H(s'',-(\frac{r}{2}bf+\xi''),-\frac{r'}{2}))$ for a general $H$,
where $r' \equiv r \mod 2l$, $s'' \equiv s \mod 2l$,
$l=\gcd(s'',\xi'',r')$,
$\xi''/l \in E_8(-1)$ is primitive
and
$r's'' \geq s'' >\langle v^2 \rangle$. 
\end{enumerate}
\end{lem}

\begin{proof}
We first note that the choice of $H$ is not important by
Proposition \ref{prop:indep}.
So we do not explain about the choice of $H$.
(1)
We set $p:=(r,\xi)$.
For  $v=(r,\frac{r}{2}bf+\xi,-\frac{s}{2})$,
we take $D \in E_8(-1)$ such that 
$ve^D=(r,\frac{r}{2}bf+\xi_1,-\frac{s'}{2})$
satisfies 
$\xi_1/p$ is primitive and
$s'>\langle v^2 \rangle$.
Since $s'=s-2(\xi,D)-r(D^2)$,
$s' \equiv s \mod 2l$.
By Proposition \ref{prop:hodge-enriques},
$e({\cal M}_H(v))=e({\cal M}_H(s',-(\frac{r}{2}bf+\xi_1),-\frac{r}{2}))$.
\begin{NB}
Since $p|\xi$ and $p|r$,
$(s',p)=(s-2(\xi,D)-r(D^2),p)=(s,p)=l$. 
\end{NB}
Since $l=(s',p)$,
we take $D_1 \in E_8(-1)$ such that
$(s',-(\frac{r}{2}bf+\xi_1),-\frac{r}{2}) e^{D_1}
=(s',-(\frac{r}{2}bf+\xi'),-\frac{r'}{2})$
satisfies $\xi'/l$ is primitive and
$r'>\langle v^2 \rangle$.
We also have $r'=r+2(\xi_1,D_1)-s'(D_1^2)
 \equiv r \mod 2l$.
Applying Proposition \ref{prop:hodge-enriques}, we have
$$
e\left({\cal M}_H 
\left(s',-(\tfrac{r}{2}bf+\xi_1),-\tfrac{r}{2}\right)\right)=
e\left({\cal M}_H 
\left(r',\tfrac{r}{2}bf+\xi',-\tfrac{s'}{2}\right)\right).
$$
(2)
For $(r',\frac{r}{2}bf+\xi',-\frac{s'}{2})$ in (1),
we take $D_2 \in E_8(-1)$ such that
$(r',\tfrac{r}{2}bf+\xi'',-\tfrac{s''}{2})=
(r',\tfrac{r}{2}bf+\xi',-\tfrac{s'}{2})e^{D_2}$
satisfies 
$\xi''/l \in E_8(-1)$ is primitive,
$s''>\langle v^2 \rangle$.
Then we have 
$$
e\left({\cal M}_H
\left(r',\tfrac{r}{2}bf+\xi',-\tfrac{s'}{2}\right)\right)
=e\left({\cal M}_H
\left(s'',-(\tfrac{r}{2}bf+\xi''),-\tfrac{r'}{2}\right)\right)
$$
by Proposition \ref{prop:hodge-enriques}.
\end{proof}

\begin{lem}\label{lem:reduction2}
For a primitive Mukai vector
$v=(r,\frac{r}{2}bf+\xi,-\frac{s}{2})$ ($b =0,-1,1, \xi \in E_8(-1)$),
there exist some zeta and $t$ such that
$$
e\left({\cal M}_H
\left(r,\tfrac{r}{2}bf+\xi,-\tfrac{s}{2}\right)\right)=
e\left({\cal M}_H
\left(2,\zeta,-\tfrac{t}{2}\right)\right)
$$
for a general $H$.
\end{lem}

\begin{proof}
(1)
We first assume that $r \equiv 0 \mod 4$ and
$s \equiv 2 \mod 4$.
By Lemma \ref{lem:reduction1}, 
we have
$$
e\left({\cal M}_H
\left(r,\tfrac{r}{2}bf+\xi,-\tfrac{s}{2}\right)\right)=
e\left({\cal M}_H
\left(r',\tfrac{r}{2}bf+\xi',-\tfrac{s'}{2}\right)\right)
$$
for a general $H$,
where $r' \equiv 0 \mod 2l$,
$s' \equiv 2 \mod 2l$, $\xi'/l \in E_8(-1)$ is primitive
and
$r' >\langle v^2 \rangle$.
For $\eta \in E_8(-1)$, we set
$D:=\sigma-\frac{(\eta^2)}{2}f+\eta$.
Then $(D^2)=0$.
Since $r \equiv 0 \mod 2l$,
we can choose $\eta$ such that
\begin{equation}
s'-rb-2=2(\xi',\eta).
\end{equation}
\begin{NB}
$\{(\xi',\eta) \mid \eta \in E_8(-1) \}={\Bbb Z}l$.
\end{NB}
Then
$(\tfrac{r}{2}bf+\xi',D)=\tfrac{r}{2}b+(\xi',\eta)=\tfrac{s'}{2}-1$
and
\begin{equation}
\begin{split}
\left(r',\tfrac{r}{2}bf+\xi',-\tfrac{s'}{2}\right)e^D=&
\left(r',\tfrac{r}{2}bf+\xi'+r'D,
-\tfrac{s'-2(\frac{r}{2}bf+\xi',D)}{2}\right)\\
=& \left(r',\tfrac{r}{2}bf+\xi'+r'D,-1 \right).
\end{split}
\end{equation}
Hence 
$$
e\left({\cal M}_H
\left(r',\tfrac{r}{2}bf+\xi',-\tfrac{s'}{2}\right)\right)=
e\left({\cal M}_H
\left(2,\zeta,-\tfrac{r'}{2}\right)\right)
$$
for a general $H$,
where $\zeta=-(\tfrac{r}{2}bf+\xi'+r'D )$.

(2)
We next assume that $r \equiv 2 \mod 4$.
If $b=0$ and $l=2$, then by using Lemma \ref{lem:reduction1} (2),
we have
$$
e\left({\cal M}_H \left(r,\xi,-\tfrac{s}{2} \right)\right)=
e\left({\cal M}_H \left(s'',-\xi'',-\tfrac{r'}{2}\right)\right)
$$
 for a general $H$.
Since $r' \equiv 2 \mod 2l$,  it is reduced to the case (1).

Assume that $b=\pm 1$ or $l=1$.
By Lemma \ref{lem:reduction1} (1),
we have
$$
e\left({\cal M}_H \left(r,\tfrac{r}{2}bf+\xi,-\tfrac{s}{2}\right)\right)=
e\left({\cal M}_H \left(r',\tfrac{r}{2}bf+\xi',-\tfrac{s'}{2}\right)\right)
$$
for a general $H$,
where $r' \equiv 2 \mod 2l$, $\xi'/l$ is primitive and 
$r' >\langle v^2 \rangle$.
Since $\xi'/l$ is primitive and $\frac{r}{2}$ is odd,
we take $\eta \in E_8(-1)$ such that
$\frac{r}{2}b+(\xi',\eta)=1$.
\begin{NB}
If $l=1$, then $\xi'$ is primitive.
If $l=2$, then $b\frac{r}{2}$ is odd.
\end{NB}
We set $D:=\sigma-\tfrac{(\eta^2)}{2}f+\eta$.
Then $(D,\tfrac{r}{2}bf+\xi')=\frac{r}{2}b+(\xi',\eta)=1$.
Hence
$$
\left(r',\tfrac{r}{2}bf+\xi',-\tfrac{s'}{2}\right)e^{(\tfrac{s'}{2}-1)D}
=\left(r',\tfrac{r}{2}bf+\xi'+r'(\tfrac{s'}{2}-1)D,-1 \right).
$$
Applying
Proposition \ref{prop:hodge-enriques},
we get
$$
e \left({\cal M}_H \left(r',\tfrac{r}{2}bf+\xi',-\tfrac{s'}{2}\right)\right)=
e \left({\cal M}_H \left(2,\zeta,-\tfrac{r'}{2}\right)\right)
$$
for a general $H$,
where 
$\zeta=-(\tfrac{r}{2}bf+\xi'+r'(\tfrac{s'}{2}-1)D)$.

(3)
Finally we assume that $r \equiv 0 \mod 4$ and $s \equiv 0 \mod 4$.
If $l=2$, then $2 \mid (\tfrac{r}{2}bf+\xi)$.
By Lemma \ref{lem:Hauzer} (2), $v$ is not primitive.  
Hence $l=1$.
By using Lemma \ref{lem:reduction1} (1) again,
we have
$$
e\left({\cal M}_H \left(r,\tfrac{r}{2}bf+\xi,-\tfrac{s}{2} \right)\right)=
e \left({\cal M}_H \left(r',\tfrac{r}{2}bf+\xi',-\tfrac{s'}{2}\right)\right)
$$
for a general $H$,
where $\xi'$ is primitive
and
$r' >\langle v^2 \rangle$.
Since we can take $\eta \in E_8(-1)$ with 
$\frac{r}{2}b+(\xi',\eta)=1$, as in the case (2), we get the claim.
\end{proof}


We shall next treat the general case.
We use induction on $r$.
We set $c_1:=d_1 \sigma+d_2 f+ \xi$, 
$\xi \in E_8(-1)$.
Replacing $v$ by $v \exp(k \sigma)$,
we may assume that
$-\frac{r}{2} <d_1 \leq \frac{r}{2}$.
We first assume that $d_1 \ne 0, \frac{r}{2}$.
We note that $(c_1,f)=d_1$.
Replacing $v$ by $v \exp(\eta)$, $\eta \in E_8(-1)$,
we may assume that $s>\langle v^2 \rangle$.
Then by Proposition \ref{prop:hodge-enriques}, 
$e({\cal M}_H(v))=e({\cal M}_H(s,-c_1,-\frac{r}{2}))$ for a general $H$.
We take an integer $k$ such that
$0<r+2d_1k \leq 2|d_1|<r$.
\begin{NB}
If $d_1>0$, then
$r=2d_1(-k)+\lambda$, $0<\lambda \leq 2d_1$.
Hence $0 <r+2d_1k \leq 2|d_1|$.
\end{NB}
Then $v \exp(kf)=(s,(-c_1+skf),-\frac{r'}{2})$,
where $r'=r+2d_1k$.
Since $s>\langle v^2 \rangle$,
Proposition \ref{prop:hodge-enriques},
implies that $e({\cal M}_{H}(s,(-c_1+skf),-\frac{r'}{2}))=
e({\cal M}_{H}(r',(c_1-skf),-\frac{s}{2}))$ for a general $H$.
By induction hypothesis, we get our claim.

If $d_1=0, \frac{r}{2}$, then we may assume that 
$-\frac{r}{2}<d_2  \leq \frac{r}{2}$.
If $d_2 \ne 0,\frac{r}{2}$, 
then we can apply the same argument and get our claim.
If $(d_1,d_2)=(0,0),(\tfrac{r}{2},0),(0,\tfrac{r}{2})$, then 
the claim follows from Lemma \ref{lem:reduction2}.

\begin{NB}
\begin{equation}
v= (r,d_1 \sigma+d_2 f+k \xi,-\tfrac{s}{2})
\end{equation}
$\xi \in E_8(-1)$ is primitive,
\begin{equation}
\begin{split}
-\frac{r}{2}<& d_1 \leq \frac{r}{2},\\
-\frac{r}{2}<& d_2 \leq \frac{r}{2},\\
0 \leq & k \leq \frac{r}{2}.
\end{split}
\end{equation}
If $d_1<\frac{r}{2}$, then
$$
e({\cal M}_H(r,d_1 \sigma+d_2 f+k \xi,-\tfrac{s}{2}))
=e({\cal M}_H(d_1 \sigma+d_2 f+k \xi,-\tfrac{s'}{2})
$$
Let $v=(r,\frac{r}{2}(\sigma+f)+k \xi,-\frac{s}{2})$
be a primitive Mukai vector such that
$0 \leq k \leq \frac{r}{2}$ and  $\xi$ is primitive. 
\end{NB}
Assume that $(d_1,d_2)=(\tfrac{r}{2},\tfrac{r}{2})$.
We may assume that 
$\xi=k \xi'$, $\xi'$ is primitive and
$0 \leq k \leq \frac{r}{2}$.

For $\eta \in E_8(-1)$,
we set $\sigma':=\sigma-\frac{(\eta^2)}{2}f+\eta$.
Then
$\sigma'$ and $f$ spans a hyperbolic lattice and
\begin{equation}
\begin{split}
\left(\tfrac{r}{2}(\sigma+f)+\xi,f \right)=& \frac{r}{2}\\
\left(\tfrac{r}{2}(\sigma+f)+\xi,\sigma' \right)=& 
\frac{r}{2}\left(1-\frac{(\eta^2)}{2}\right)+(\xi,\eta).
\end{split}
\end{equation}
Replacing $\eta$ by $-\eta$ if necessary,
we can take $\eta$ such that
\begin{equation}
(\xi',\eta)=
\begin{cases}
-1, & 2  \mid (\eta^2)/2\\
1, & 2 \nmid (\eta^2)/2.
\end{cases}
\end{equation}

\begin{equation}
\frac{r}{2}\left(1-\frac{(\eta^2)}{2}\right)+(\xi,\eta)
\equiv 
\begin{cases}
\frac{r}{2}-k \mod r & 2 \mid (\eta^2)/2\\
k \mod r & 2 \nmid (\eta^2)/2.
\end{cases}
\end{equation} 
If $k \ne \frac{r}{2},0$, then
we can reduce to the case where $|d_1|<\frac{r}{2}$.
If $k=0$, then choosing $\eta$ with $(\eta^2)=-2$,
we can reduced to the case $d_1=0$.
If $k=\frac{r}{2}$, then
we choose $\eta$ satisfying
$((\xi'-\eta)^2) \equiv ({\xi'}^2)+2 \mod 4$.
Then  
\begin{equation}
\frac{r}{2}\left(1-\frac{(\eta^2)}{2}\right)+\frac{r}{2}(\xi',\eta)
\equiv 
0 \mod r.
\end{equation}
Hence we can also reduce to the case where $d_1=0$.
Therefore Theorem \ref{thm:e-poly} holds.
\qed

\begin{rem}
In \cite{Hauzer}, Hauzer takes a hyperbolic lattice
spanned by $\sigma$ and $\sigma+f+e_1$, where $e_1 \in E_8(-1)$
is a $(-2)$-vector.
Then $c_1=(r+(\xi,e_1))\sigma'+(r/2)f'+\xi'$.

\end{rem}

\begin{NB}
\begin{thm}
For $r,s \in 2{\Bbb Z}$ and $L \in \NS(X)$,
assume that $v:=(r,c_1(L),-\tfrac{s}{2})$ is primitive.
Let ${\cal M}_H(r,L,-\frac{s}{2})$ be the substack of 
${\cal M}_H(v)$ consisting of $E$ with $\det E=L$.
Then
${\cal M}_H(r,L,-\tfrac{s}{2}) \ne \emptyset$ if and only if
\begin{enumerate}
\item
$\gcd(r,c_1(L),s)=1$ and $\langle v^2 \rangle \geq -1$ or 
\item 
$\gcd(r,c_1(L),s)=2$ and $\langle v^2 \rangle \geq 2$ 
or 
\item
$\gcd(r,c_1(L),s)=2$,
$\langle v^2 \rangle=0$ and $L \equiv \frac{r}{2}K_X \mod 2$.
\end{enumerate}
\end{thm}
\end{NB}

By Theorem \ref{thm:e-poly},
Theorem \ref{thm:exist:intro} for $r>0$ is reduced to
the following claim.
\begin{prop}[{Kim \cite{Kim4}}]\label{prop:Kim}
Assume that $v:=(2,c_1(L),-\tfrac{s}{2})$ is primitive.
Then
${\cal M}_H(2,L,-\tfrac{s}{2}) \ne \emptyset$ if and only if
\begin{enumerate}
\item
$\gcd(2,c_1(L))=1$ and $\langle v^2 \rangle \geq -1$ or 
\item 
$\gcd(2,c_1(L))=2$ and $\langle v^2 \rangle \geq 2$ 
or 
\item
$\gcd(2,c_1(L))=2$,
$\langle v^2 \rangle=0$ and $L \equiv K_X \mod 2$.
\end{enumerate}
\end{prop}

\begin{NB}
Indeed if $\gcd(r,c_1(L),s)=2$, then
$L$ is changed to
$L+\langle v,v(K_X) \rangle K_X=
L+\frac{s-r}{2}K_X=L+K_X$
under the $(-1)$-reflection by $v({\cal O}_X)$.
Thus Proposition \ref{prop:hodge-enriques} is refined as
\begin{equation}\label{eq:refine}
e({\cal M}_H(r,L,-\tfrac{s}{2}))=
e({\cal M}_H(s,-(L+K_X),-\tfrac{r}{2}))
\end{equation}
for a general $H$,
where $(c_1(L)^2)<0$. 
We note that $r \mod 4$ and $s \mod 4$ are invariant under 
the multiplication by a line bundle.
We set $L:=L'+\frac{r}{2}K_X$. 
Since $r+s \equiv 2 \mod 4$,
the equation \eqref{eq:refine} is 
$$
e({\cal M}_H(r,L'+\tfrac{r}{2}K_X,-\tfrac{s}{2}))=
e({\cal M}_H(s,-L'+\tfrac{s}{2}K_X,-\tfrac{r}{2})).
$$
\end{NB}
For the case of Proposition \ref{prop:Kim} (iii),
by using Proposition \ref{prop:hodge-enriques} (2),
we have Theorem \ref{thm:exist:intro} (iii). 
In the next subsection, we shall give another proof
of Kim's result. 

\begin{NB}

\begin{prop}\label{prop:00/08/17}
We set $G_1:={\cal E}_{|\{x \} \times X}$ and
$G_2:={\cal E}_{|X \times \{x \} }$.
Assume that $\deg_{G_1}(v)=0$ and 
$l(v):=-\langle v,\varrho_X \rangle/\rk v_1>0$, 
$a(v):=\langle v,v_1 \rangle/\rk v_1>0$. 
Let $\varepsilon$ be an element of $K(X) \otimes {\Bbb Q}$
such that $v(\varepsilon) \in
v_1^{\perp} \cap \varrho_X^{\perp}$,
$|\langle v(\varepsilon)^2 \rangle| \ll 1$ and
$(H,c_1(\varepsilon))=0$. 
Then ${\cal H}_{\cal E}$ induces an isomorphism
\begin{equation}
{\cal M}_H^{G_1+\varepsilon}(v)^{ss} \to 
{\cal M}_{\widehat{H}}^{G_2+\widehat{\varepsilon}}
(-{\cal H}_{\cal E}(v))^{ss},
\end{equation}
where $\widehat{\varepsilon}={\cal H}_{\cal E}(\varepsilon)$.
\end{prop}

\begin{proof}
Since $H$ is general with respect to $v(E_0)$,
we see that ${\cal E}_{|\{x \} \times X}$ is $G_1$-twisted
stable. Then we see that 
Lemma \ref{lem:G^2-2} holds.
We next show that Lemma \ref{lem:G^0-2} holds.
We may assume that $E$ is $\mu$-stable. 
If $E^{\vee \vee} \ne E_0, E_0(K_X)$, 
then $\Hom(E,E_0)=\Hom(E,E_0(K_X))=0$.
If $E^{\vee \vee}=E_0, E_0(K_X)$, then
$\Hom(E,{\cal E}_{|\{x \} \times X})=0$ for
$x \in X \setminus \Supp(E^{\vee}/E)$.
Thus Lemma \ref{lem:G^0-2} holds.
Then the same proof of Theorem \ref{thm:00/08/17}
works and we get our claim.
\end{proof}

\end{NB}

\subsection{Relative Fourier-Mukai transform}\label{subsect:relative-FM}

For $G \in K(X)$ with $\rk G>0$, 
we define $G$-twisted semi-stability replacing
the Hilbert polynomial $\chi(E(nH))$ by the $G$-twisted Hilbert
polynomial $\chi(G^{\vee} \otimes E(nH))$.
${M}_H^G(r,L,-\frac{s}{2})$ 
denotes the moduli scheme of $G$-twisted
semi-stable sheaves $E$ with $v(E)=(r,c_1(L),-\frac{s}{2})$ and
$\det E=L$.
\begin{NB}
For $L \in \NS(X)$ with $c_1(L)=c_1$,
we also set
${M}_H^G(r,L,-\frac{s}{2})=\{E \in 
{M}_H^G(r,c_1,-\frac{s}{2}) \mid \det E=L\}$.
\end{NB}
If $G={\cal O}_X$, then
we also denote ${M}_H^G(r,L,-\frac{s}{2})$ by
${M}_H(r,L,-\frac{s}{2})$. 
The $G$-twisted semi-stability is the same as the $\alpha$-twisted
semi-stability, where $\alpha=c_1(G)/\rk G$.

We have an elliptic fibration $X \to {\Bbb P}^1$
such that $2f$ is the divisor class of a fiber.   
Let $G_1$ be a locally free sheaf on $X$
such that $v(G_1)=v({\cal O}_X)+v({\cal O}_X(\sigma))+(0,0,k)$.
We set $Y:={M}_{H+nf}^{G_1}(0,2f,1)$, where $H$ is an ample
divisor on $X$ and $n \geq 0$. 
Then $\chi(G_1,E)=-\langle v(G_1),v(E) \rangle=0$
for $E \in {M}_{H+nf}^{G_1}(0,2f,1)$.
\begin{lem}
$Y$ consists of $G_1$-twisted stable sheaves.
\end{lem}

\begin{proof}
If $E \in {M}_{H+nf}^{G_1}(0,2f,1)$ is properly $G_1$-twisted
semi-stable, then there is a proper subsheaf $E_1$ of $E$
such that $\chi(G_1,E_1)=0$ and $E/E_1$ is also purely 1-dimensional.
We set $v(E_1)=(0,\xi_1,a)$, $a \in {\Bbb Z}$.
Then $(\xi_1,c_1(G_1))=2a \in 2{\Bbb Z}$.  
Since $(c_1(E_1),c_1(G_1)),(c_1(E/E_1),c_1(G_1)) \geq 0$
and $(c_1(E),c_1(G_1))=2$,
$(c_1(E_1),c_1(G_1))=0$ or 
$(c_1(E/E_1),c_1(G_1))=0$. 
If every singular fiber is irreducible,
then $(c_1(E_1),c_1(G_1))>0$ and 
$(c_1(E/E_1),c_1(G_1))>0$.
Therefore $Y$ consists of $G_1$-twisted stable sheaves.
\end{proof}
By
\cite{Br:1}, $Y$ is a smooth projective surface which is a 
compactification of $\Pic_{X/C}^1$.
Hence $Y \cong X$.
Let ${\cal E}$ be a universal family.
Let $\Psi:{\bf D}(X) \to {\bf D}(Y)$
be a contravariant Fourier-Mukai transform
defined by
\begin{equation}
\Psi(E):={\bf R}\Hom_{p_Y}(p_X^*(E),{\cal E}),
\end{equation}
where $p_X$ and $p_Y$ are the projections from $X \times Y$ to $X$ and $Y$
respectively.

Let $L_1$ be a line bundle on $C \in |H|$
and set $G_2:=\Psi(L_1)[1]$ 
(see the above of \cite[Lem.\ 3.2.3]{PerverseII}).
We also set $\widehat{H}:=-c_1(\Psi(G_1))$
(\cite[Lem.\ 3.2.1]{PerverseII}).
\begin{prop}[{\cite[Prop.\ 3.4.5]{PerverseII}}]\label{prop:elliptic}
Assume that $(c_1(L),f)=\frac{r}{2} \in {\Bbb Z}$ and 
$\chi(E,L_1)<0$.
$\Psi$ induces an isomorphism
$$
{\cal M}_{H+nf}^{G_1}\left(r,L,-\frac{s}{2}\right) \cong
{\cal M}_{\widehat{H}+nf}^{G_2}\left(0,D,-\frac{s'}{2}\right)$$
 for $n \gg 0$,
where 
$D$ is an effective divisor such that
$(D^2)=(c_1(L)^2)+rs$ and $(D,2f)=r$.
\end{prop}

\begin{rem}
Replacing $E$ by $E(mf)$ ($m \gg 0$),
$\chi(E,L_1)<0$ holds.
\end{rem}

\begin{rem}
Although $G_1$ is fixed, $H$ is not fixed. 
So we can change $H$ to be general.
\end{rem}

\begin{cor}\label{cor:elliptic}
Assume that $2 \nmid c_1(L)$. Then 
${\cal M}_{H+nf}^{G_1}(2,L,-\frac{s}{2}) \ne \emptyset$ if and only if 
$(c_1(L)^2)+2s \geq 0$.
\end{cor}

\begin{proof}
Let $D$ be the divisor in Proposition \ref{prop:elliptic}.
Since $(D^2)=(c_1(L)^2)+2s$, we shall prove that
the condition is $(D^2) \geq 0$.
Obviously the condition is necessary.
Conversely assume that $(D^2) \geq 0$.  
Since $Y \cong X$ is unnodal,
$|D|$ contains a reduced and irreducible
curve $C$ by \cite[Thm. 3.2.1]{CD}, where we also use 
$(D,f)=1$ if $(D^2)=0$.
Then a line bundle $F$ on $C$ with $\chi(F)=-\tfrac{s'}{2}$
is a member of ${\cal M}_{\widehat{H}+nf}^{G_2}(0,D,-\tfrac{s'}{2})$.
\end{proof}

\begin{NB}
\begin{rem}
If $(c_1^2)-2s=0$, then
$D$ is the support of a multiple
fiber of an elliptic fibration.
\end{rem}
\end{NB}

\subsection{Rank 2 case}\label{subsect:unnodal:rank2}

\begin{prop}\label{prop:exist1}
Assume that $2 \nmid c_1(L)$ is primitive.
Then ${\cal M}_H(2,L,-\frac{s}{2}) \ne \emptyset$ for
a general $H$ if and only if 
$(c_1(L)^2)+2s \geq 0$.
\end{prop}

\begin{proof}
If $2 \nmid (c_1(L),f)$ or
$2 \nmid (c_1(L),\sigma)$, then
the claim follows from Corollary \ref{cor:elliptic}.
Otherwise we may assume that $c_1(L) \in E_8(-1)$ and $c_1(L)$ is primitive.
Then there is $\eta \in E_8(-1)$ with
$(c_1(L),\eta)=1$.
We set $\sigma':=\sigma-\frac{(\eta^2)}{2}f+\eta$.
Then ${\Bbb Z}\sigma' +{\Bbb Z}f$ spans a hyperbolic lattice
and $(\sigma',c_1(L))=1$.
Since $X$ is unnodal
and $f$ is effective, $\sigma'$ is effective and
$2\sigma'$ defines an elliptic fibration.
Therefore the claim also holds for this case.  
\end{proof}

\begin{prop}\label{prop:exist2}
Assume that $2 \mid c_1(L)$.
Then ${\cal M}_H(2,L,-\frac{s}{2}) \ne \emptyset$
if and only if 
\begin{enumerate}
\item[(i)]
$(c_1(L)^2)+2s>0$ or 
\item[(ii)]
$(c_1(L)^2)+2s=0$ and $L \equiv K_X \mod 2$. 
\end{enumerate}
\end{prop}

\begin{proof}
We may assume that $L=0, K_X$.
If there is a stable sheaf $E$, then
$E \cong E(K_X)$ and $(c_1(L)^2)+2s \geq -2$, or
$E \not \cong E(K_X)$ and $(c_1(L)^2)+2s \geq -1$.
Since $4 \mid s$,
$(c_1(L)^2)+2s =2s \geq 0$.

Assuming $(c_1(L)^2)+2s >0$, we first prove 
${\cal M}_H(2,L,-\frac{s}{2}) \ne \emptyset$
for a general $H$. 
We set $k:=\frac{s}{4}>0$.
Then 
$E_1 \oplus E_2$ with 
$v(E_1)=(1,0,-k-\frac{1}{2}), v(E_2)=(1,0,-k+\frac{1}{2})$
belongs to the moduli stack
${\cal M}_H(2,L,-\frac{s}{2})^{\mu\text{-}ss}$
of $\mu$-semi-stable sheaves.
Let ${\cal F}(v_1,v_2)$ be the substack of
${\cal M}_H(2,L,-\frac{s}{2})^{\mu\text{-}ss}$ consisting
of $E$ whose Harder-Narasimhan filtration
$0 \subset F_1 \subset F_2=E$ satisfies
$v(F_1)=v_1$ and $v(F/F_1)=v_2$.
Then 
\begin{equation}
\begin{split}
\dim {\cal F}(v_1,v_2)= &\langle v_1,v_2 \rangle+\dim {\cal M}_H(v_1)
+\dim {\cal M}_H(v_2)\\
=& \langle v^2 \rangle-\langle v_1,v_2 \rangle.
\end{split}
\end{equation} 
We set 
$v_1=(1,\xi_1,-\frac{s_1}{2})$, 
$v_2=(1,\xi_2,-\frac{s_2}{2})$.
Then $\xi_1$ and $\xi_2$ are numerically trivial, 
$s_1<s_2$ and $s_1+s_2=s$.
Then
$\langle v_1,v_2 \rangle=\frac{s_1+s_2}{2}=\frac{s}{2}>0$.
By the deformation theory,
each irreducible component ${\cal M}$
of ${\cal M}_H(v)^{\mu\text{-}ss}$
satisfies $\dim {\cal M}\geq \langle v^2 \rangle$.
Hence there is a stable sheaf.

We next treat the case where 
$(c_1(L)^2)+2s=0$.
By \cite{Y:twist1}, $M_H(2,K_X,0) \cong X$
and $E(K_X) \cong E$ for all $E \in M_H(2,K_X,0)$.
Moreover there is a universal family 
which defines a Fourier-Mukai transform.
Then for a stable sheaf $E$ with 
$v(E)\equiv v \mod K_X$,
we see that $E \in M_H(2,K_X,0)$.     
In particular, $M_H(2,0,0)=\emptyset$.
\end{proof}

Therefore Proposition \ref{prop:Kim} holds by
Proposition \ref{prop:exist1},\ \ref{prop:exist2},
and we complete the proof of
Theorem \ref{thm:exist:intro} for $r>0$.

\begin{rem}
Nuer constructed $\mu$-stable vector bundles of rank 4 in
\cite[Thm. 5.1]{N}.
This reuslt (\cite[Thm. 5.1]{N})
does not follow from our method.
\end{rem}

\subsection{Rank 0 case}\label{subsect:unnodal:rank0}
We shall prove Theorem \ref{thm:exist:intro}
for $r=0$.
We first note that if 
${\cal M}_H(0,L,-\tfrac{s}{2}) \ne \emptyset$, then
$L$ is effective.
For the proof of Theorem \ref{thm:exist:intro},
we use Proposition \ref{prop:elliptic}. 
By choosing a suitable elliptic fibration,
we may assume that $(c_1(L),f)>0$.
Then we have
 $$
e\left({\cal M}_H \left(0,L,-\frac{s}{2}\right)\right)
=e\left({\cal M}_H \left(r,L',-\frac{s'}{2}\right)\right),
$$ where
$(c_1(L'),2f)=r$.
Then the case of $r=0$ is reduced to the case of $r>0$
at least
for $\gcd(c_1(L),s)=1$ or 
$(c_1(L)^2) >0$. 
\begin{NB}
Let $2F_A,2F_B$ be multiple fibers of $X \to {\Bbb P}^1$.
If $r$ is odd, then
$rF_A+K_X=rF_B$, since $K_X=rK_X=rF_A-rF_B$.
Hence 
${\cal M}_H(0,rF_A+K_X,-s/2)$ consists of stable vector bundles
on $F_B$.
\end{NB}
Assume that $\gcd(c_1(L),s)=2$ and $(c_1(L)^2)=0$.
Then ${\cal M}_H(0,L,-\frac{s}{2}) = \emptyset$
or ${\cal M}_H(0,L+K_X,-\frac{s}{2}) = \emptyset$.
If $L \equiv 0 \mod 2$, then
 there is $\frac{r}{2}C \in |L|$ such that
$C$ is a smooth fiber of the elliptic fibration, and
a stable vector bundle $F$ of rank $\frac{r}{2}$ and $\chi(F)=-\frac{s}{2}$
on $C$ is a member of ${\cal M}_H(0,L,-\frac{s}{2})$.
Hence 
${\cal M}_H(0,L,-\frac{s}{2}) \ne \emptyset$
if and only if $L \equiv 0 \mod 2$ as we claimed in 
Theorem \ref{thm:exist:intro}.

\begin{rem}\label{rem:rank0}
It is easy to see that \cite[Thm.\ 1.7]{Stability}
holds for Enriques surfaces.
Indeed a similar claim to \cite[Prop.\ 2.7]{Stability}
(see Appendix) holds and \cite[Prop.\ 2.8, Prop.\ 2.11]{Stability}
hold if we modify the number $N$ in the claims suitably.
 
Then Theorem \ref{thm:exist:intro} for $r=0$ can also 
be reduced to the claim for $r>0$.
\end{rem}

\begin{rem}\label{rem:r=0}
Since $X$ is unnodal, effectivity implies 
$(c_1(L)^2) \geq 0$ and $(c_1(L),H)>0$.
Conversely if $(c_1(L)^2) \geq 0$ and $(c_1(L),H)>0$, then
$L$ is effective by the Riemann-Roch theorem.
\end{rem}

\section{A nodal case}\label{sect:nodal}

We shall treat the nodal case by adding a deformation argument
and results of Kim \cite{Kim1} and \cite{Kim:excep}.
\begin{thm}\label{thm:exist:nodal}
Let $X$ be a nodal Enriques surface over ${\Bbb C}$.
We take $r,s \in {\Bbb Z}$ $(r>0)$ and $L \in \NS(X)$ such
that $r-s$ is even.
Assume that 
$\gcd(r,c_1(L),\frac{r-s}{2})=1$, i.e., the Mukai vector
is primitive.
Then
${\cal M}_H(r,L,-\tfrac{s}{2}) \ne \emptyset$ for a general $H$
if and only if
\begin{enumerate}
\item
$\gcd(r,c_1(L),s)=1$ and $(c_1(L)^2)+rs \geq -1$ or 
\item 
$\gcd(r,c_1(L),s)=2$ and $(c_1(L)^2)+rs \geq 2$ 
or 
\item
$\gcd(r,c_1(L),s)=2$,
$(c_1(L)^2)+rs =0$ and $L \equiv \frac{r}{2}K_X \mod 2$ or
\item
$(c_1(L)^2)+rs =-2$,
$L \equiv D+\frac{r}{2}K_X \mod 2$, where 
$D$ is a nodal cycle, i.e.,
$D$ is effective, $(D^2)=-2$ and $|D+K_X| =\emptyset$.
\end{enumerate}
\end{thm}

\begin{rem}
If $(c_1(L),H')>0$ for an ample divisor $H'$, then
the same claim holds for $r=0$.
\end{rem}

\begin{NB}
For $L$ satisfying the condition,
${\cal M}_H(2,L+K_X,0)$ consists of
$\pi_*(F)$, where
$F$ is a line bundle on the covering K3 surface.
Since $v(F)=(1,C,0)$ with $(C^2)=-2$,
$(C,\pi^*(H'))>0$ implies $C$ is effective.
Then $\pi_*(F/{\cal O}_{\widetilde{X}})$
is a 1-dimensional sheaf with
$\det \pi_*(F/{\cal O}_{\widetilde{X}})=\det F(-K_X)=
{\cal O}_X(L)$. Thus $L$ is effective.
\end{NB}

Obviously $(c_1(L)^2)+rs \geq -2$ is necessary for the
non-emptyness of the moduli stack.
We first assume that 
$(c_1(L)^2)+rs \geq -1$.
In his case, the existence is a consequence of
Theorem \ref{thm:exist:intro}.
Let $(X,H)$ be an Enriques surface $X$ and an ample divisor
$H$ on $X$.
By \cite[Prop.\ 1.4.1]{CD},
$H^1(X,T_X) \cong {\Bbb C}^{\oplus 10}$ and $H^2(X,T_X)=0$.
We also have $H^2(X,{\cal O}_X)=0$.
Hence a polarized deformation of the pair
$(X,H)$ is unobstructed.
\begin{NB}
If we embed $X$ into a projective space ${\Bbb P}^n$ 
by a complete linear system
$|mH|$, then
the Hilbert scheme is unobstructed at $X$ and
the map of the tangent spaces is surjective:
We first note that $H^1(X,{\cal O}_{{\Bbb P}^n})=0$
by Kodaira vanishing theorem and the ampleness of
$mH-K_X$.
By the Euler sequence
\begin{equation}
0 \to {\cal O}_X \to {\cal O}_X(1) \to (T_{{\Bbb P}^n})_{|X} \to 0
\end{equation}
and $H^1(X,{\cal O}_X)=H^2(X,{\cal O}_X)=0$,
$H^1(X,(T_{{\Bbb P}^n})_{|X})=0$.
Then by the exact sequence
\begin{equation}
0 \to T_X \to (T_{{\Bbb P}^n})_{|X} \to N_{X/{{\Bbb P}^n}} \to 0,
\end{equation}  
we have a long exact sequence
\begin{equation}
H^0(X,N_{X/{{\Bbb P}^n}}) \to H^1(X,T_X) \to H^1(X,(T_{{\Bbb P}^n})_{|X})
\to H^1(X,N_{X/{{\Bbb P}^n}}) \to H^2(X,T_X). 
\end{equation}
Hence $H^1(X,N_{X/{{\Bbb P}^n}})=0$ and the map of tangent
spaces is surjective.
\end{NB}
Let $({\cal X}, {\cal H}) \to S$ be a general
deformation of $(X,H)$ such that
a general member is not nodal and
$({\cal X}_0,{\cal H}_0)=(X,H)$ $(0 \in S)$.
Then we have a family of moduli spaces of semi-stable
sheaves
$f:{M}_{({\cal X},{\cal H})}(v) \to S$.
Under the assumption (i), (ii), (iii) in Theorem \ref{thm:exist:intro},
${M}_{({\cal X},{\cal H})}(v)_s \ne \emptyset$
for unnodal ${\cal X}_s$.
\begin{NB}
We don't need the genericity of ${\cal H}_s$.
\end{NB}
Hence $f$ is dominant.
By the projectivity of $f$, 
$\im f=S$. Hence 
${M}_{({\cal X},{\cal H})}(v)_s \ne \emptyset$
for all $s$.
\begin{prop}
Let $X$ be an Enriques surface.
Under the conditions (i), (ii), (iii) of Theorem \ref{thm:exist:intro},
${\cal M}_H(r,L,-\frac{s}{2}) \ne \emptyset$
for a general $H$. 
\end{prop}
If $\gcd(r,c_1(L),a)=2$, $(c_1(L)^2)+rs=0$ and
$L \not \equiv \frac{r}{2}K_X  \mod 2$,
then ${\cal M}_H(r,L,-\frac{s}{2}) =\emptyset$.
Indeed since $M_H(r,L+K_X,-\frac{s}{2}) (\ne \emptyset)$ 
is an Enriques surface
for a general $H$ and the universal family 
induces a Fourier-Mukai transform,
we see that every stable sheaf $E$ with 
$v(E)=(r,c_1(L),-\frac{s}{2})$ belongs to
$M_H(r,L+K_X,-\frac{s}{2})$.
Therefore Theorem \ref{thm:exist:nodal}
holds if $(c_1(L)^2)+rs \geq -1$.

\begin{rem}
If $r$ is odd and $H$ is general, then $\Ext^2(E,E)=0$ for
$E \in {\cal M}_H(r,c_1,-\frac{s}{2})$.
In this case, $f$ is a smooth morphism in a neighborhood of $0$. 
\begin{NB}
If $E \cong E(K_X)$ and simple (which occurs for even rank cases), 
then $\Ext^2(E,E(K_X)) \cong {\Bbb C}$.
Since $\tr:\Ext^2(E,E) \to H^2(X,{\cal O}_X)$
is a 0-map, the obstruction for infinitesimal deformation
may be non-zero. 
So $E$ may not deform under the deformation of $X$ in general.
\end{NB}
\end{rem}

\begin{NB}
\begin{prop}[{Kim \cite[Thm. 1]{Kim:excep}}]
Assume that 
$r$ is even and $(c_1(L)^2)-rs=-2$.
${\cal M}_H(r,L,-\frac{s}{2}) \ne \emptyset$
for a general $H$ if and only if 
$L \equiv N+\frac{r}{2}K_X \mod 2$, where 
$N$ is a nodal cycle.
\end{prop}

\begin{proof}
If there is $E \in {\cal M}_H(r,L,-\frac{s}{2})$, then
$E \cong E(K_X)$. Hence $L \equiv N+\frac{r}{2}K_X$ by
\cite[Thm. 1 (1)]{Kim:excep}.
We next prove the converse direction.
This is also contained in the proof of
\cite[Thm. 1 (2)]{Kim:excep}, although 
a technical assumption $\gcd(\frac{r},(H,c_1(L)),-\frac{s}{2})$
is added. 
For completeness, we give an outline of the proof. 
Assume that $L$ satisfies
$L \equiv N+\frac{r}{2}K_X \mod 2$, where 
$N$ is a nodal cycle.
Then we have a decomposition
$\pi^*(N)=N_1+N_2$ such that $N_1$ and $N_2$ are nodal cycles on 
$\widetilde{X}$, 
$N_2=\iota^*(N_1)$,
$N_1 \cap N_2=0$.
Hence $\pi^*(v)=(r,\pi^*(N)+2\pi^*(M),-s)
=(\tfrac{r}{2},N_1+\pi^*(M),-\frac{s}{2})+
(\tfrac{r}{2},N_2+\pi^*(M),-\frac{s}{2})$.
We set $w:=(\tfrac{r}{2},N_1+\pi^*(M),-\frac{s}{2})$.
Since $\langle w^2 \rangle=-2$,
there is a semi-stable sheaf $E$ with $v(E)=w$ with respect to
$\pi^*(H)$.
Then $\pi_*(E)$ is also semi-stable with respect to $H$.
\begin{NB2}
$\pi^*(\pi_*(E))=E \oplus \iota^*(E)$ and
the Hilbert polynomial of $E$ and $\iota^*(E)$ are the same.
Hence $\iota^*(E)$ and $\pi^*(\pi_*(E))$ are semi-stable
with respect to $\pi^*(H)$.
Hence $\pi_*(E)$ is semi-stable. 
\end{NB2}
\end{proof}

\begin{lem}
Let $C$ be an effective divisor on $\widetilde{X}$
such that $\iota(C) \cap C =\emptyset$.
Then $\det \pi_*({\cal O}_C)=\pi_*(C)$. 
\end{lem}

Let $D$ be an effective divisor of $(D^2)=-2$ such that
$D+\iota^*(D)=\pi^*(y)$, $(y^2)=-2$.
Then $Z:=\frac{1}{2}\pi_*(D+\iota^*(D))$ 
is an effective divisor such that 
$(Z^2)=-2$.
If $D+\iota^*(D)=\sum_i n_i C_i$, then
$Z=\frac{1}{2}\sum_i \pi_*(C_i)$.
Since $n_i \iota_*(C_i)=n_j C_j$ for a $j$,
$Z$ is an integral divisor such that
$\pi_*(D)=y$. Since $C_i$ are smooth rational curves,
$\pi_*(C_i)$are smooth rational curves.
Since $D$ is simple normal crossing,

 \end{NB} 

We treat the remaining case, i.e., $(c_1(L)^2)+rs=-2$.
This case is completely studied by Kim in \cite{Kim1} and \cite{Kim:excep}.
For completeness of the proof, 
we add an outline of the proof in \cite{Kim:excep}.
Let $\pi:\widetilde{X} \to X$ be the universal cover of $X$.
$\widetilde{X}$ is a K3 surface.
We need the following elementary fact. 
\begin{lem}\label{lem:determinant}
For a locally free sheaf $F$ of rank $r$ on $\widetilde{X}$,
$$
\det \pi_*(F) \cong \det(\pi_*(\det F))((r-1)K_X).
$$
\end{lem}

\begin{proof}
Let $H$ be an ample divisor on $X$.
Since $\pi^*(H)$ is ample,
we have an exact sequence
\begin{equation}
0 \to {\cal O}_{\widetilde{X}}(-n \pi^*(H))^{\oplus (r-1)}
\to F \to I_Z(D) \to 0,
\end{equation} 
where $D$ is a divisor, $Z$ is a 0-dimensional subscheme of 
$\widetilde{X}$ and $n$ is sufficiently large.
Since $\pi_*({\cal O}_{\widetilde{X}})={\cal O}_X \oplus {\cal O}_X(K_X)$
and ${\cal O}_{\widetilde{X}}(D-(r-1)n\pi^*(H))=\det F$,
we get the claim.
\begin{NB}
Since
$\pi_*(\det F)=\pi_*({\cal O}_{\widetilde{X}}(D))(-(r-1)nH)$,
$\det \pi_*(\det F)=(\det \pi_*({\cal O}_{\widetilde{X}}(D)))(-2(r-1)nH)$.
\end{NB} 
\end{proof}

\begin{NB}
\begin{lem}
Let $D$ be a divisor on $\widetilde{X}$.
Then the following three conditions are equivalent.
\begin{enumerate}
\item[(1)]
$(D^2)=(D,\iota^*(D))-2$
\item[(2)]
$\pi_*({\cal O}_{\widetilde{X}}(D))$ is
a stable vector bundle with 
$\langle v(\pi_*({\cal O}_{\widetilde{X}}(D)))^2 \rangle=-2$.
\item[(3)] 
$\langle v(\pi_*({\cal O}_{\widetilde{X}}(D)))^2 \rangle=-2$
and $\det \pi_*({\cal O}_{\widetilde{X}}(D))={\cal O}_X(C+2A+K_X)$,
where $C$ is a nodal cycle.
\end{enumerate}
\end{lem}

\begin{proof}
For a divisor $D$ on $\widehat{X}$,
$\pi^*(\pi_*({\cal O}_{\widehat{X}}(D))) \cong
{\cal O}_{\widehat{X}}(D) \oplus {\cal O}_{\widehat{X}}(\iota^*(D))$.
Hence $\langle v(\pi_*({\cal O}_{\widetilde{X}}(D)))^2 \rangle=-2$
if and only if 
$\langle v({\cal O}_{\widehat{X}}(D)),v({\cal O}_{\widehat{X}}(\iota^*(D)))
\rangle=0$.
This condition is equivalent to
$(D^2)=(D,\iota^*(D))-2$.
If $D$ satisfies the condition, then
${\cal O}_{\widehat{X}}(D) \not \cong
{\cal O}_{\widehat{X}}(\iota^*(D))$.
$\pi_*({\cal O}_{\widehat{X}}(D))$ is $\mu$-semi-stable
with respect to all ample divisor $H$ on $X$.
Since $v(\pi_*({\cal O}_{\widehat{X}}(D)))$ is primitive,
if $\pi_*({\cal O}_{\widehat{X}}(D))$ is properly $\mu$-semi-stable
with respect to an ample divisor $H$, then
$\pi_*({\cal O}_{\widehat{X}}(D))$ is not $\mu$-semi-stable
for an ample divisor $H'$ in a neighborhood of $H$.
Therefore $\pi_*({\cal O}_{\widehat{X}}(D))$ is 
$\mu$-stable for any ample divisor $H$.

Assume that
$\det \pi_*({\cal O}_{\widetilde{X}}(D))={\cal O}_X(C+2A+K_X)$.
Since $C$ is simply connected,
$\pi^{-1}(C)$ is a disjoint union of nodal curves
$C_1,C_2$ with $\iota(C_1)=C_2$.
then $D+\iota^*(D)=C_1+C_2+2\pi^*(A)$
$\pi_*({\cal O}_{\widetilde{X}}(C_1+\pi^*(A)))$ is a stable
vector bundle with $\langle v(\pi_*({\cal O}_{\widetilde{X}}(C_1+\pi^*(A))))^2
\rangle=-2$.
Hence $v(\pi_*({\cal O}_{\widetilde{X}}(C_1+\pi^*(A))))=
v(\pi_*({\cal O}_{\widetilde{X}}(D)))$.
Then $\pi_*({\cal O}_{\widetilde{X}}(C_1+\pi^*(A)))\cong
\pi_*({\cal O}_{\widetilde{X}}(D))$ by the stabilities and
$\chi(\pi_*({\cal O}_{\widetilde{X}}(C_1+\pi^*(A))),
\pi_*({\cal O}_{\widetilde{X}}(D)))=2$.
Hence $D=C_1+\pi^*(A),C_2+\pi^*(A)$, which implies
(1).
\end{proof}
\end{NB}

\begin{NB}
Let $v=(r,\xi,a)$ be a Mukai vector
such that $\xi=C+2A+\frac{r}{2}K_X$ and
$\langle v^2 \rangle=-2$, where $C$ is a nodal cycle.
For $w=(2,C+2A+K_X,\frac{ra}{2})$,
$\langle w^2 \rangle=-2$.
Hence there is a stable vector bundle $E$ with
$v(E)=w$.
Then $E \cong E(K_X)$ and we have
$E=\pi_*({\cal O}_{\widetilde{X}}(D))$
for a divisor $D$ on $\widetilde{X}$.
We have $(D^2)=(D,\iota^*(D))-2$.
For $u:=(\frac{r}{2},D,a)$, 
$\langle u^2 \rangle=-2$.
Let $F$ be a stable locally free sheaf with $v(F)=u$.
Then $\pi_*(F)$ is a stable locally free sheaf with $v(\pi_*(F))=v$.
\end{NB}

We also need the following result of Kim \cite[Thm. 1]{Kim:excep}.
\begin{lem}\label{lem:exceptional}
Assume that $r \in 2{\Bbb Z}_{>0}$, $a \in {\Bbb Z}$ and
$L \in \NS(X)$ satisfy
$(c_1(L)^2)-2ra=-2$. Then
${\cal M}_H(r,L,a)\ne \emptyset$ for a general $H$
if and only if ${\cal M}_H(2,L-(\frac{r}{2}-1)K_X,\frac{ra}{2}) \ne \emptyset$.
\end{lem}

\begin{proof}
Since the formulation of the claim is slightly different
from \cite[Thm.\ 1]{Kim:excep},
we write the proof.
We set $v:=(r,c_1(L),a)$.
Since $\gcd(r,c_1(L))=1$, there is an ample divisor $H$ with
$\gcd(r,(c_1(L),H))=1$.
Indeed we first take a divisor $\eta$ with
$\gcd(r,(c_1(L),\eta))=1$.
Then we have an ample divisor $H=\eta+r\lambda$,
$\lambda \in \Amp(X)$, which satisfies the claim.
We may prove the claim for this polarization.

For $E \in {\cal M}_H(r,L,a)$, 
we have $E \cong E(K_X)$. By the proof of \cite[Lem.\ 1.12]{Ta},
there is a simple vector bundle $F$ such that $E=\pi_*(F)$.
Since $E$ is rigid, 
$F$ is also rigid (see the proof of \cite[Thm.\ 1]{Kim:excep}).
\begin{NB}
If $E \cong E \otimes K_X$, then
there is an equivariant isomorphism
$\phi:(\pi^*(E),\iota) \to (\pi^*(E),-\iota)$.
Thus we have a commutative diagram
\begin{equation}
\begin{CD}
\pi_*(\pi^*(E)) @>{\phi}>> \pi_*(\pi^*(E))\\
@V{\iota^*}VV @VV{-\iota^*}V\\
\pi_*(\pi^*(E)) @>{\phi}>> \pi_*(\pi^*(E)).
\end{CD}
\end{equation}
Then $\phi^2$ is an equivariant isomorphism of
$(\pi^*(E),\iota)$.
Since $E$ is simple,
$\phi^2=\lambda$, $\lambda \in k$.
Replacing $\phi$ by $\sqrt{\lambda}^{-1}\phi$,
we assume that $\phi^2=1$.
Then we have an eigen space decomposition 
$\pi_*(\pi^*(E))=E_1 \oplus E_{-1}$.
By the action of $\iota^*$,
$E_1$ and $E_2$ are exchanged.
Therefore $E_1$ and $E_2$ define subsheaves $F_1$ and $F_2$
of $\pi^*(E)$ and $\iota(F_1)=F_2$.
\end{NB}
By the stability of $E$,
$F$ is stable with respect to $\pi^*(H)$.
We have $\pi^*(E) \cong F \oplus \iota^*(F)$.
We set $C:=\det(F)$.
Then $v(F)=(\frac{r}{2},C,a)$ and
$C+\iota^*(C)=\pi^*(L)$.
We see that $(C^2)=(C,\iota^*(C))-2$
and $(C^2)-ra=-2$.
\begin{NB}
That is $\langle v(F)^2 \rangle=\langle (\frac{r}{2},C,a)^2 \rangle=-2$.
\end{NB}
We set $E':=\pi_*({\cal O}_{\widetilde{X}}(C))$.
Since $\pi^*(E) \cong {\cal O}_{\widetilde{X}}(C) \oplus
{\cal O}_{\widetilde{X}}(\iota^*(C))$,
we see that $v(E')=(2,c_1(L),\frac{(C^2)}{2}+1)=
(2,c_1(L),\frac{ra}{2})$. 
By Lemma \ref{lem:determinant},
$$
\det E'=(\det E)\left(-\left(\frac{r}{2}-1 \right)K_X \right)
={\cal O}_X \left(L-\left(\frac{r}{2}-1 \right)K_X \right).
$$
Obviously $E'$ is semi-stable with respect to
$H$. Since $\gcd(r,(c_1(E'),H))=1$, it is $\mu$-stable.
\begin{NB}
Another argument:
Obviously $E'$ is $\mu$-semi-stable
with respect to all ample divisor $H'$ on $X$.
Since $v(E')$ is primitive,
if $E'$ is properly $\mu$-semi-stable
with respect to an ample divisor $H'$, then
$E'$ is not $\mu$-semi-stable
for an ample divisor $H''$ in a neighborhood of $H'$.
Therefore $E'$ is 
$\mu$-stable for any ample divisor $H'$.
\end{NB}
Therefore ${\cal M}_H(2,L-(\frac{r}{2}-1)K_X,\frac{rs}{2}) \ne \emptyset$.

Conversely for 
$E' \in {\cal M}_H(2,L-(\frac{r}{2}-1)K_X,\frac{ra}{2})$,
there is a divisor $C$ with $\pi_*({\cal O}_{\widetilde{X}}(C))=E'$.
Since $\pi^*(E') \cong {\cal O}_{\widetilde{X}}(C) \oplus
{\cal O}_{\widetilde{X}}(\iota^*(C))$,
we see that $(C^2)=(C,\iota^*(C))-2$ and $(C^2)+2=ra$.
For $u:=(\frac{r}{2},C,a)$, we have
$\langle u^2 \rangle=-2$.
By $\gcd(r,(c_1(L),H))=1$ and $\pi^*(L)=C+\iota^*(C)$,
$\gcd(r,(C,\pi^*(H)))=1$.
Let $F$ be a $\mu$-stable locally free sheaf such that $v(F)=u$
with respect to $\pi^*(H)$.
Then $E:=\pi_*(F)$ is a $\mu$-stable locally free sheaf with $v(\pi_*(F))=v$.
By Lemma \ref{lem:determinant},
$\det E=L$. Therefore ${\cal M}_H(r,L,a)\ne \emptyset$.
\end{proof}

\begin{prop}\label{prop:nodal:exceptional}
Assume that $r \in {\Bbb Z}_{\geq 0}$, $s \in {\Bbb Z}$ and
$L \in \NS(X)$ satisfy
$r \equiv s \mod 2$ and
$(c_1(L)^2)+rs=-2$.
If $r=0$, then we further assume that $(c_1(L),H')>0$
for an ample divisor $H'$ on $X$. Then
${\cal M}_H(r,L,-\frac{s}{2})\ne \emptyset$ for a general $H$
if and only if $L=D+2A+\frac{r}{2}K_X$, where 
$D$ is a nodal cycle and $A \in \NS(X)$.  
\end{prop}

\begin{proof}
If $r>0$, then the claim is a consequence of
Lemma \ref{lem:exceptional} and \cite[Thm.\ 3.4]{Kim1}. 
If $r=0$, then the claim is a consequence of Remark \ref{rem:rank0}
(see also Corollary \ref{cor:appendix:exceptional}).
\end{proof}

\begin{NB}
If $((D+2A)^2)=-2$ and $(D+2A,H)>0$, then
$D+2A$ is effective.
$\pi^*(D)=C_1+C_2$.
Then $((C_1+\pi^*(A))^2)=-2$ and
$(C_1+\pi^*(A),\pi^*(H))>0$.
Hence we have an exact sequence
$$
0 \to {\cal O}_{\widetilde{X}} \to
{\cal O}_{\widetilde{X}}(C_1+\pi^*(A)) \to F \to 0
$$ 
where is a 0-dimensional sheaf.
Applying $\pi_*$,
we have an exact sequence
$$
0 \to {\cal O}_X \oplus K_X \to
\pi_*({\cal O}_{\widetilde{X}}(C_1+\pi^*(A)))
\to \pi_*(F) \to 0.
$$
Since $\det \pi_*({\cal O}_{\widetilde{X}}(C_1+\pi^*(A)))
={\cal O}_X(D+2A+K_X)$,
$D+2A$ is effective.

For $(0,D,a)$ with $(D^2)=-2$,
there is an ample divisor $H$ such that 
$(D,H)=1$.
Multiplying ${\cal O}_{X}(nH)$ to
$(0,D,a)$, we assume that $a=1$.

\end{NB}

\section{Appendix}\label{sect:appendix}
Let $X$ be any Enriques surface and $H$ be an ample divisor on $X$. 
For $\omega=tH$, $t>0$,
let $Z_{(0,\omega)}:{\bf D}(X) \to {\Bbb C}$ be a stability function
defined by
\begin{equation}
Z_{(0,\omega)}(E):=\langle e^{\omega \sqrt{-1}},v(E) \rangle,\;
E \in {\bf D}(E).
\end{equation}
Let ${\frak T}_{(0,\omega)}$ be the full subcategory 
of $\Coh(X)$ generated by torsion sheaves and torsion free
stable sheaves $E$ with
$Z_{(0,\omega)}(E) \in {\Bbb H} \cup {\Bbb R}_{<0}$.
Let ${\frak F}_{(0,\omega)}$ be the full subcategory 
of $\Coh(X)$ generated by torsion free
stable sheaves $E$ with
$-Z_{(0,\omega)}(E) \in {\Bbb H} \cup {\Bbb R}_{<0}$.
Let ${\frak A}_{(0,\omega)} (\subset {\bf D}(X))$ be the category generated by
${\frak T}_{(0,\omega)}$ and ${\frak F}_{(0,\omega)}[1]$.
If$(\omega^2) \ne 1$, then
$\sigma{(0,\omega)}:=({\frak A}_{(0,\omega)},Z_{(0,\omega)})$
is a stability condition.
${\frak A}_{(0,\omega)}$ is constant
on $(\omega^2) \ne 1$.
We set 
\begin{NB}
\begin{equation}
{\frak A}_{(0,\omega)}:=
\begin{cases}
{\frak A}^\mu, & (\omega^2)>1\\
{\frak A}, & (\omega^2)<1.
\end{cases}
\end{equation}
\end{NB}
\begin{defn}
\begin{enumerate}
\item[(1)]
For $(\omega^2)>1$, we set
${\frak T}^\mu:={\frak T}_{(0,\omega)}$,
${\frak F}^\mu:={\frak F}_{(0,\omega)}$ and
${\frak A}^\mu:={\frak A}_{(0,\omega)}$.
\item[(2)]
For $(\omega^2)<1$, we set
${\frak T}:={\frak T}_{(0,\omega)}$,
${\frak F}:={\frak F}_{(0,\omega)}$ and
${\frak A}:={\frak A}_{(0,\omega)}$.
\end{enumerate}
\end{defn}
For $E \in {\frak F}^\mu$, 
we have an exact sequence
\begin{equation}
0 \to E_1 \to E \to E_2 \to 0
\end{equation} 
such that
\begin{enumerate}
\item[(1)]
 $E_1$ is generated by ${\cal O}_X$ and $K_X$, and
\item[(2)]
 $E_2 \in {\frak F}^\mu$ satisfies
$\Hom({\cal O}_X,E_2)=\Hom({\cal O}_X(K_X),E_2)=0$, i.e.,
$E_2 \in {\frak F}$.
\end{enumerate}
Since $H^1({\cal O}_X(K_X))=H^1({\cal O}_X)=0$,
$$
E_1 \cong {\cal O}_X^{\oplus n} \oplus {\cal O}_X(K_X)^{\oplus m}.
$$
We also have
$E_1=\Hom({\cal O}_X,E) \otimes {\cal O}_X \oplus \Hom(\cal O_X(K_X),E) 
\otimes {\cal O}_X(K_X)$.
For $E \in {\frak T}$, the natural homomorphism
$$
\phi:E \to \Hom(E,{\cal O}_X)^{\vee} \otimes {\cal O}_X \oplus
\Hom(E,{\cal O}_X(K_X))^{\vee} \otimes {\cal O}_X(K_X) 
$$
is surjective and
$\ker \phi \in {\frak T}^\mu$.

We set
$$
{\cal E}:=\ker({\cal O}_X \boxtimes {\cal O}_X 
\oplus {\cal O}_X(K_X)^{\vee} \boxtimes {\cal O}_X(K_X)
\to {\cal O}_{\Delta}).
$$
As in \cite{MYY:2011:1},
$\Phi_{X \to X}^{{\cal E}^{\vee}[1]}:{\bf D}(X) \to
{\bf D}(X)$
induces an isomorphism
${\frak A} \to {\frak A}^\mu$
\begin{NB}
${\cal O}_X,K_X$ are irreducible objects of ${\cal A}$.
${\cal E}_{|\{x \} \times X}[1]$ is an irreducible objects
of ${\cal A}$.
$k_x$ is not irreducible in ${\frak A}$:
\begin{equation}
0 \to {\cal O}_X \oplus K_X \to k_x \to  {\cal E}_{|\{x \} \times X}[1]
\to 0.
\end{equation}
\end{NB}
and we have a commutative diagram
\begin{equation}
\begin{CD}
{\frak A} @>{\Phi_{X \to X}^{{\cal E}^{\vee}[1]}}>> {\frak A}^\mu\\
@V{Z_{(0,\omega)}}VV @VV{Z_{(0,\omega')}}V\\
{\Bbb C} @<<{\times (\omega^2)}< {\Bbb C}
\end{CD}
\end{equation}
where $\omega'=\omega/(\omega^2)$.
In particular, we get the following.
\begin{prop}
$\Phi_{X \to X}^{{\cal E}^{\vee}[1]}$ induces an isomorphism
\begin{equation}
{\cal M}_{(0,\omega)}\left(r,\eta+\frac{r}{2}K_X,-\frac{s}{2}\right)
\cong {\cal M}_{(0,\omega')}\left(s,\eta+\frac{s}{2}K_X,-\frac{r}{2}\right).
\end{equation}
\end{prop}

Applying Toda's argument to the wall crossing along 
the line $\omega=tH$, $t>0$,
we get the following result (see also the argument in \cite{MYY:2011:1}).

\begin{prop}[cf. Toda {\cite{Toda}}]
\begin{enumerate}
\item[(1)]
If $(\omega^2) \gg 0$ and $(\eta,\omega)>0$, then
$$
{\cal M}_{(0,\omega)}\left(r,\eta+\frac{r}{2}K_X,-\frac{s}{2}\right)=
{\cal M}_\omega \left(r,\eta+\frac{r}{2}K_X,-\frac{s}{2}\right).
$$
\item[(2)]
$e\left({\cal M}_{(0,\omega)}\left(r,\eta+\tfrac{r}{2}K_X,
-\tfrac{s}{2}\right)\right)$
is independent of a general choice of $\omega$.
\end{enumerate}
\end{prop}


\begin{rem}
Wall crossing along the line $\omega=tH$ is very similar to
the classical wall crossing
of Gieseker semi-stability, since
${\frak A}_{(0,\omega)}$ is almost the same.
\end{rem}

\begin{cor}\label{cor:appendix:exceptional}
$$
e\left({\cal M}_H\left(r,\eta+\frac{r}{2}K_X,-\frac{s}{2}\right)\right)=
e\left({\cal M}_{H}\left(s,\eta+\frac{s}{2}K_X,-\frac{r}{2}\right)\right)
$$
for a general $H$.
\end{cor}

We have another proof of Proposition \ref{prop:nodal:exceptional}. 
\begin{prop}\label{prop:appendix:exceptional}
Assume that $(\eta^2)+rs=-2$.
${\cal M}_H(r,\eta+\tfrac{r}{2}K_X,-\tfrac{s}{2}) \ne \emptyset$
for a general $H$
if and only if $\eta \equiv D \mod 2$,
where $D$ is a nodal cycle.
\end{prop}

\begin{NB}
Let $D$ be an effective divisor with $(D^2)=-2$.
Since $H^1({\cal O}_X)=H^2({\cal O}_X)=0$,
by the exact sequence
$$
0 \to {\cal O}_X(-D) \to {\cal O}_X \to {\cal O}_D \to 0,
$$
we have $H^1({\cal O}_D) \cong H^2({\cal O}_X(-D)) \cong
H^0({\cal O}_X(D+K_X))^{\vee}$.
\end{NB}

\begin{proof}
By the proof of Theorem \ref{thm:e-poly},
we have
$e({\cal M}_H(r,\eta+\tfrac{r}{2}K_X,-\tfrac{s}{2}))=
e({\cal M}_H(2,\eta'+K_X,-\tfrac{s'}{2}))$,
where $\eta \equiv \eta' \mod 2$.
By \cite{Kim1}, ${\cal M}_H(2,\eta'+K_X,-\tfrac{s'}{2}) \ne \emptyset$
if and only if $\eta \equiv D \mod 2$, where
$D$ is a nodal cycle.
Therefore the claim holds. 
\end{proof}

\begin{NB}
Let $v_0$ be a Mukai vector with $\langle v_0^2 \rangle=-1$.
For a general $H$,
${\cal M}_H(l v_0)=
\{E_0^{\oplus n} \oplus E_0(K_X)^{\oplus m} \mid n+m=l \}$,
where $E_0$ is a stable sheaf with $v(E_0)=v_0$.
\end{NB}


\begin{thebibliography}{00}


\bibitem{Br:1}
T.\ Bridgeland,
{\it Fourier-Mukai transforms for elliptic surfaces,}
J. Reine Angew. Math. {\bf 498} (1998), 115--133. 
\bibitem{Br:3}
T.\ Bridgeland,
{\it Stability conditions on K3 surfaces,}
Duke Math. J.  {\bf 141} (2008), 241--291, math.AG/0307164.

\bibitem{CD}
F.\ Cossec and I.\ Dolgachev,
{\it Enriques surfaces. I.}
Progress in Mathematics, 76.
Birkh\"{a}user Boston, Inc., Boston, MA,1989.

\bibitem{Hauzer}
M.\ Hauzer,
{\it On moduli spaces of semistable sheaves on
Enriques surfaces,}
Ann. Polon. Math. {\bf 99} (2010), no. 3, 305--321, arXiv:1003.5857.



\bibitem{Kim1}
H.\ Kim,
{\it Exceptional bundles on nodal Enriques surfaces,}
Manuscripta Math. {\bf 82}  (1994),  no. 1, 1--13. 
\bibitem{Kim2}
H.\ Kim, 
{\it Moduli spaces of stable vector bundles on Enriques surfaces,}
Nagoya Math. J. {\bf 150}  (1998), 85--94.
\bibitem{Kim3}
H.\ Kim,
{\it Moduli spaces of bundles mod Picard groups on some elliptic surfaces,}
Bull. Korean Math. Soc.  {\bf 35}  (1998),  no. 1, 119--125. 
\bibitem{Kim:excep}
H.\ Kim,
{\it Exceptional bundles of higher rank and rational curves,}
Bull. Korean Math. Soc.  {\bf 35}  (1998),  no. 1, 149--156.
\bibitem{Kim4}
H.\ Kim,
{\it Stable vector bundles of rank 2 on Enriques surfaces,}
J. Korean Math. Soc. {\bf 43} (2006), 765--782.
\bibitem{MW}
K.\ Matsuki and R.\ Wentworth,
{\it Mumford-Thaddeus principle on the moduli space of vector bundles
 on an algebraic surface,}
Internat. J. Math. {\bf 8} (1997), no. 1, 97--148.

\bibitem{MYY:2011:1}
H.\ Minamide, S.\ Yanagida and K.\ Yoshioka,
{\it The wall-crossing behavior 
for Bridgeland's stability conditions
on abelian and K3 surfaces,}
J. Reine Angew. Math. to appear
DOI: 10.1515/crelle-2015-0010.
\bibitem{N}
H.\ Nuer, 
{\it A note on the existence of stable vector bundles onEnriques surfaces,}
arXiv:1406.3328.
\bibitem{N2}
H.\ Nuer, 
{\it Projectivity and Birational Geometry of 
Bridgeland Moduli spaces on an Enriques Surface,}
 arXiv:1406.0908.
\bibitem{Sacca}
G.\ Sacca,
{\it Relative compactified Jacobians of linear systems on Enriques
surfaces,} arXiv:1210.7519 v2.  
\bibitem{Ta}
F.\ Takemoto,
{\it Stable vector bundles on algebraic surfaces II,}
Nagoya Math. J. {\bf 52} (1973), 173--195.

\bibitem{Toda}
Y.\ Toda,
{\it Moduli stacks and invariants of semistable objects on K3 surfaces,}
 Adv. Math {\bf 217} (2008), no. 6, 2736--2781,
arXiv:math/0703590.

\bibitem{Yamada}
K.\ Yamada,
{\it Singularities and Kodaira dimension 
of moduli scheme of stable sheaves on Enriques surfaces,}
Kyoto J. Math. {\bf 53}  (2013),  no. 1, 145--153. 


\bibitem{chamber}
K.\ Yoshioka,
\emph{Chamber structure of polarizations and the 
moduli of stable sheaves on a ruled surface},
Internat.~J.~Math. {\bf 7} (1996), 411--431.



\bibitem{Y:twist1}
K.\ Yoshioka,
\emph{Twisted stability and Fourier-Mukai transform I},
Compositio~Math. {\bf 138} (2003), 261--288.

\bibitem{Y:twist2}
K.\ Yoshioka,
{\it Twisted stability and Fourier-Mukai transform II,}
Manuscripta Math. {\bf 110} (2003), 433--465.


\bibitem{Stability}
K.\ Yoshioka,
{\it Stability and the Fourier-Mukai transform II,}
Compositio Math. {\bf 145} (2009), 112--142. 


\bibitem{PerverseII}
K.\ Yoshioka,
\emph{Perverse coherent sheaves and Fourier-Mukai transforms on surfaces II},
Kyoto J. Math. {\bf 55} (2015), 365--459.





\end{thebibliography}
\end{document}